\documentclass[11pt,onecolumn]{article}
\usepackage{amscd}
\usepackage{times}
\usepackage{amsmath}
\usepackage{amssymb}
\usepackage{xspace}
\usepackage{theorem}
\usepackage{graphicx}
\usepackage{ifpdf}
\usepackage{url,hyperref}
\usepackage{latexsym}
\usepackage{euscript}
\usepackage{xspace}
\usepackage[all]{xy}
\usepackage{color}
\usepackage{makeidx}
\usepackage{tikz}
\long\def\remove#1{}
\newtheorem{theorem}{Theorem}[section] 

\newtheorem{obs}[theorem]{Observation}
\newtheorem{corollary}[theorem]{Corollary}

\newtheorem{definition}[theorem]{Definition}
\newtheorem{proposition}[theorem]{Proposition}
\newenvironment{proof}{{\em Proof:}}{\hfill{\hfill\rule{2mm}{2mm}}}

\newcommand {\mm}[1] {\ifmmode{#1}\else{\mbox{\(#1\)}}\fi}

\newcommand{\img}{\mathrm img}
\newcommand{\supp}{\mathrm supp}

\newcommand{\rank}                {\mm {\rm rank}}
\renewcommand{\L}                        {\mathcal {L}}

\newcommand{\cancel}[1]

\begin{document}

\title{Refinement of Novikov--Betti numbers and of Novikov homology provided by an angle valued map}

\author{
Dan Burghelea  \thanks{
Department of Mathematics,
The Ohio State University, Columbus, OH 43210,USA.
Email: {\tt burghele@math.ohio-state.edu}}
}
\date{}

\maketitle

\hskip 1.5in  dedicated to  the memory of Yuri Petrovich
Solovyov 

\begin{abstract}

To a pair (X,f), X compact ANR and $f:X\to \mathbb S^1$ a continuous  angle valued map, $\kappa$ a field  and  a nonnegative integer $r,$ one assigns  a  finite configuration  of complex numbers $z$   with  multiplicities $\delta^f_r(z)$  
and a  finite configuration of free $\kappa[t^{-1}, t]$--modules  $\hat \delta^f_r$ of  rank $\delta^ f_r(z)$   
indexed by the same numbers $z.$   This is in analogy with the configuration of eigenvalues and of generalized eigen-spaces of a linear operator in a finite dimensional complex vector space. 
The configuration $\delta^f_r$  refines the Novikov--Betti number in dimension $r$   and the configuration $\hat \delta^f_r$ refines  the Novikov homology in dimension $r$ associated with the  cohomology class defined by $f$.

In the case  the field  $\kappa= \mathbb C$ the configuration $\hat \delta^f_r$ provides by ''von-Neumann completion"  a configuration $\hat{\hat \delta}^f_r$ of mutually orthogonal closed  Hilbert submodules of the $L_2$--homology of the infinite cyclic cover of $X$ determined by the map $f,$ which is an $L^\infty(\mathbb S^1)$-Hilbert module.

\end{abstract}

\thispagestyle{empty}
\setcounter{page}{1}

%
\tableofcontents

\vskip .2in

\section {Introduction}

  In \cite {BH} and \cite {Bu} for a pair (X,f), $X$ compact ANR\footnote {cf subsection \ref{SS21} for definition}, $f$ a real  valued map, $\kappa$ a field and $r$ a nonnegative integer we have assigned
 a   configuration  $\delta^f_r$ of complex numbers  $z$  with  multiplicities $\delta^f_r(z)$  and 
 a   configuration of vector spaces  $\hat \delta^f_r(z),$ indexed by the same set $\{z\mid \delta^f_r(z)\geq 1\}=\supp (\delta^f_r),$  with the following properties:
 \begin{enumerate}
\item  $\dim \hat \delta^f_r(z)=  \delta^f_r(z),$  
\item $\sum_{z\in \supp (\delta^f_r)} \delta^f_r(z)= \beta_r(X;\kappa)$, $\beta_r(X;\kappa)= \dim H_r(X;\kappa),$
\item $\oplus _{z\in \supp (\delta^f_r)} \hat \delta^f_r(z) \simeq H_r(X;\kappa).$ \footnote {$\simeq$ denotes isomorphism}
\end{enumerate}
The integers $\beta_r(X;\kappa)$  are referred to as the Betti numbers  of $X$ with coefficients in $\kappa.$ 

Each vector space $\hat\delta^f_r(z)$ appears as a quotient of subspaces of $H_r(X;\kappa)$ well separated in a sense specified later in the paper.
This  assignment was  in analogy with the {\it spectral package}  of a pair $(V,T),$ $V$ a f.d. complex vector space, $T:V\to V$ a linear map, to which  one assigns the configurations 
$\delta^T$ and $\hat \delta^T$ defined by the  eigenvalues of $T$  with their multiplicity and  
the collection of  generalized eigenspaces corresponding to the eigenvalues of $T.$  
It was shown in \cite {BH}  that :
\begin{itemize}
\item {\bf P1}: The  assignment $f\rightsquigarrow \delta^f_r$ is continuous,
\item {\bf P2}: For $M^n$ a closed topological manifold, one has  $\delta^f_r(z)= \delta^f_{n-r}(i \overline z),$ 
\item {\bf P3}: If $X$ is homeomorphic to a simplicial complex or a  Hilbert cube manifold then for an open and dense set of maps $f,$ one has
 $\delta^f_r(z)\leq 1.$
\end{itemize}
When  $\kappa= \mathbb C$ a Hermitian scalar product on $H_r(X;\mathbb C)$ (i.e. a Hilbert space structure on $H_r(X;\mathbb C)$)  provides a canonical realization of the vector spaces $\hat \delta^f_r(z)$ as a collection of mutually orthogonal subspaces $\hat {\hat \delta}^f_r(z)\subseteq H_r(M;\mathbb C)$  such that {\bf P1}, {\bf P2} and {\bf P3} above continue to hold for $\hat{\hat\delta}^f_r.$ 
The configuration $\delta^f_r$ can be reformulated as a {\it characteristic polynomials}  $P_r^f(z),$  a monic polynomial whose zeros are the complex numbers $z$ in the support of $\delta^f_r$ with multiplicity  $\delta^f_r(z),$  of degree equal to the Betti number in dimension $r.$

The complex numbers $z$ in the support of $\delta^f_r$ with multiplicity  $\delta^f_r(z),$ equivalently 
the zeros of $P^f_r(z)$ with their multiplicities, 
can be calculated in case $X$ is a simplicial complex and $f$ a simplicial map by effective algorithms 
and are part of the {\it level persistence invariants} or  {\it bar codes},  exactly those relevant for  the global topology of $X.$ They are of 
interest in data analysis.  
The configurations $\hat{\hat \delta}^f_r,$ when considered for a closed Riemannian manifold, make  a real valued map $f$  a provider  of an orthogonal decomposition (depending on $f$) of  the space of harmonic forms on the Riemannian manifold. In particular, for a generic $f,$  the components have dimension one providing a {\it canonical base} in spaces of harmonic forms.
\vskip .2in 

In  this paper we present a  similar picture  for a pair $(X,f),$  $X$ a compact ANR  and  $f$ an angle valued map $f:X\to \mathbb S^1,$ with similar virtues. In this case an additional homological  data, the cohomology class $\xi^f\in H^1(X; \mathbb Z)$  determined by the map $f,$ is involved.  As expected  the results are similar but  more subtle and more complex.
 
In the case of an angle valued map  one has to replace  the Betti numbers  $\beta_r(X; \kappa)$ by  the Novikov--Betti numbers  $\beta^B(X,\xi;\kappa)$  and  the $\kappa-$vector space $H_r(X;\kappa) $ by the Novikov homology $H^N_r(X,\xi; \kappa[t^{-1},t]). $  In this paper the Novikov homology is a free $\kappa[t^{-1}, t]$-module derived from $H_r(\tilde X; \kappa),$ the homology of the infinite cyclic cover 
$\tilde X$ associated with $\xi$  and $\beta^N_r(X,\xi;\kappa)= \rm{rank}\ H_r^N(X,\xi;\kappa[t^{-1},t])$. One produces analogous configurations $\delta^f_r$ and $\hat \delta^f_r$ which satisfy properties 1., 2., 3. and (analogues of) {\bf P1}, {\bf P2}, {\bf P3} above.
In case of $\kappa= \mathbb C,$ instead of a Hilbert space structure on $H_r(X;\mathbb C)$ one  considers  a Hilbert module structure on the von-Neumann completion of $H^N_r(X,\xi; \kappa[t^{-1},t]),$ see subsection \ref{SS21}, which always exists . This permits to convert the configuration $\hat \delta ^f_r$ into a configuration of closed Hilbert submodules  $\hat{\hat\delta}^f_r$ 
with 
{\bf P1}, {\bf P2} and {\bf P3} continuing to hold as in the case of real valued maps. 
To formulate the results more precisely one needs first to recall for the reader some algebra and algebraic topology concepts and establish some notations.   
\vskip .1in 

Let $\kappa$ be a field and $\kappa[t^{-1}, t]$ be the ring of Laurent polynomials with coefficients in $\kappa.$  This ring is  a commutative 
$\kappa$--algebra, an integral domain and a principal ideal domain. As a consequence, for any f.g. $\kappa[t^{-1}, t]-$module $M,$ the quotient $F(M):= M/ T(M)$ with $T(M)$  the  submodule of torsion elements, is a f.g. free module.  The f.g. module $T(M)$ being torsion is a finite dimensional vector space over $\kappa.$ The only  invariant of $F(M)$ is its rank which, when $Q$ is a field containing $\kappa[t^{-1},t],$ is equal to the dimension of the $Q$--vector space $M\otimes_{\kappa[t^{-1}, t]} Q= F(M)\otimes_{\kappa[t^{-1}, t]} Q.$

For a pair $(X,\xi\in H^1(X;\mathbb Z))$ one considers the  infinite cyclic cover  $\pi:\tilde X\to X$  associated to $\xi,$  and the deck transformation $\tau:\tilde X\to \tilde X,$  the restriction of the free action $\mu: \mathbb Z\times \tilde X\to \tilde X$ associated  to $\xi$
 to  $1\times \tilde X.$
The $\kappa-$vector space $H_r(\tilde X;\kappa)$ equipped with  the isomorphism $t_r:H_r(\tilde X;\kappa)\to H_r(\tilde X;\kappa)$ induced by $\tau$ becomes a $\kappa [t^{-1}, t]$--module.  When $X$ is a compact ANR this module is finitely generated.  For any $Q$ commutative ring which contains $\kappa[t^{-1}, t]$  denote by $H^N_r(X,\xi; Q):= F(H_r(\tilde X,\kappa))\otimes _{\kappa[t^{-1}, t])} Q$ and refer to this free module as {\it Novikov homology } with coefficients in $Q.$ 
When $Q=\kappa[t^{-1}, t]],$ the field of Laurent power series, the Novikov homology is a $\kappa [t^{-1}, t]]-$vector space which is exactly what Novikov has considered in his approach to Morse theory for an angle valued map. In this paper we consider the case $Q= \kappa[t^{-1}, t].$

The rank of $H^N_r(X,\xi; Q)$  is independent of $Q,$ denoted by $b^N_r(X,\xi;\kappa),$ and  referred to as the Novikov--Betti number in dimension $r$ w.r. to $\kappa.$ 

Suppose $\kappa= \mathbb C.$ The ring $\mathbb C[t^{-1}, t]$ can be completed to a finite von Neumann algebra 
$\mathcal N,$ 
see subsection \ref{SS21} or \cite{Lu}.  This algebra is exactly $L^\infty (\mathbb S^1).$  A f.g free $\mathbb C[t^{-1}, t]$--module $M$ equipped with a $\mathbb C[t^{-1}, t]$-inner product can be completed to an $\mathcal N$-Hilbert  module $\overline M$ of finite type, and a collection of split submodules of $M$ to closed Hilbert submodules of $\overline M$;  therefore a collection $N_\alpha$ of quotients of split  submodules of $M$ provides, in a canonical manner, a collection of closed Hilbert submodules $\overline N_\alpha$ of $\overline M$ as described in subsection \ref{SS21}. 
Different $\mathbb C[t^{-1}, t]$-inner products on $M$ provide isometric Hilbert modules completions so one can ignore the inner product in the notation $\overline M.$    

\vskip .1in 
In this paper we start with a pair $(X, \xi),$ $X$ a compact ANR, $\xi \in H^1(X;\mathbb Z)$ and a field $\kappa.$ 
Consider the free $\kappa[t^{-1},t]$--module $ H^N_r(X,\xi;\kappa[t^{-1},t]).$ 
When $\kappa= \mathbb C,$ by using the {\it von Neumann completion} described in subsection \ref{SS21}, a $\mathbb C [t^{-1}, t]$--inner product  on $H_r^N(X,\xi; \mathbb C [t^{-1}, t])$  permits  to pass to the $\mathcal N$-Hilbert module  $\overline {H_r^N(X,\xi; \mathbb C[t^{-1}, t])}$  and to convert the configurations $\hat{\delta}^f_r$  into configurations of  mutually orthogonal Hilbert  submodules of $\overline {H_r^N(X,\xi; \mathbb C [t^{-1}, t])}.$

Note that the $\mathcal N$-Hilbert module $\overline {H_r^N(X,\xi; \mathbb C[t^{-1}, t])}$ is isomorphic to the $\mathcal N$-Hilbert module $H_r^{L_2}(\tilde X)$ known as $L_2$-homology defined in case $X$ is a closed Riemannian manifold  using the Riemannian metric, or in case $X$ is a finite CW-complex using the cell-structure of $X.$  

The main result of the paper is the following theorem. 

\begin{theorem}\label {T1}\
 \begin{enumerate}
 \item To a continuous map $f:X\to \mathbb S^1$ one can alssociate  a monic polynomial $P^f_r(z)$ with non vanishing  roots of  degree equal to $b^N_r(X,\xi^f;\kappa),$  equivalently  a configuration $\delta^f_r$ of non vanishing  complex numbers $z$ with multiplicities $\delta^f_r(z) \geq 1,$   
$$\mbox {Zeros of}\  P_r(z)\ni z \rightsquigarrow \delta^f_r(z)\in  \mathbb Z_{\geq 1},$$ 
which satisfies {\bf P1}, {\bf P2} and {\bf P3}.  

\item One can refine the configuration $\delta^f_r$ to  the assignment 
$$ \mbox {Zeros of} \  P_r(z)\ni z  \rightsquigarrow (L'(z) \subset L(z)) \in \mathcal S (M), \ M= H^N(X,\xi; \kappa[t^{-1},t]) $$
with $\mathcal S(M)$ the set of pairs of free split submodudeles of $M$ with $L'\subset L$ 
such that:
\begin{enumerate}
\item $\bigoplus \hat\delta^f_r(z)$ is isomorphic to $H_r^N(X,\xi; \kappa[t^{-1}, t])$ where $\hat \delta ^f_r (z)= L_r(z)/ \L_r'(z)$ and 
\item $\rank( \hat \delta^f_r(z))= \delta^f_r(z).$ 
\end{enumerate}
\item 
In case $\kappa= \mathbb C$ and a $\mathbb C[t^{-1},t]$-inner product \footnote {see the definition in subsection \ref{SS21}} on $H_r^N(X,\xi; \mathbb C[t^{-1}, t])$ is given, by using  von-Neumann completion,  one can convert the assignment above (the configuration $\hat \delta^f_r$) into a configuration $\hat{\hat \delta}^f_r(z)$ of mutually orthogonal closed Hilbert submodules of the $L^\infty(\mathbb S^1)$-Hilbert module $\overline {H_r^N(X,\xi; \mathbb C[t^{-1}, t])}$ which satisfies {\bf P1} and {\bf P2}.  Up to an isometry of  Hilbert module structures  the configurations $\hat{\hat \delta}^f_r$ are independent of the  inner product.
\end{enumerate}       
\end{theorem}
We view the configuration $\delta^f_r$ as a refinement of the Novikov-Betti number $b_r^N(X,\xi; \kappa)$   and the configuration $\hat\delta^f_r$ as a refinement of the Novikov homology $H_r^N(X,\xi; \mathbb C[t^{-1}, t])$ or better said as an additional structure on the Novikov homology. 
In this paper we give only the construction of the configurations $\delta^f_r, \hat\delta^f_r$ and $\hat{\hat \delta}^f_r;$
The details  of the proof that the configurations $\hat{\hat \delta}^f_r$  satisfy {\bf P1}, {\bf P2} and {\bf P3} will be presented in a paper  in preparation \cite {Bu2}. The configuration $\delta^f_r$  was introduced and studied in \cite{BH}, 
however the configurations $\hat\delta ^f_r$ and $\hat{\hat \delta}^f_r$ have not been considered before.
 \vskip .1in

We note that the configuration $\delta^f_r,$ or equivalently the polynomial $P_r^f(z),$ i.e. the roots with their multiplicities can be explicitly calculated in case $X$ is a finite simplicial complex and $f$ a simplicial map.
Precisely one can produce algorithms with input the simplicial complex and the values of $f$ on vertices and output  the zeros of the polynomial $P_r^f(z)$ as a pair of real number (the real and the imaginary part) with their multiplicities. 
Such algorithm  is presented in \cite {BD11} where the zeros of $P^f_r(z)$ appear as "closed $r$-bar codes" and "open $(r-1)$-bar codes" 
\footnote {The zero $z= \rho e^{i\theta} \in \mathbb C\setminus 0$  represents a closed bar code when $\rho \geq 1$ and an open bar code when $\rho <1$ }.This algorithm  uses a different  definition of the configuration $\delta^f_r$ based on bar codes in "level persistence"
cf \cite {CSD09} or \cite {BD11} . 

We also note that for $\kappa= \mathbb C$ and for $X$ an $n$--dimensional  closed Riemannian manifold the space of $L_2$-harmonic differential forms of degree  $n-r$ on the complete Riemannian manifold $\tilde X$ identifies to the $L_2$-homology in dimension $r$ of $\tilde X$ via the Hodge theory on Riemannian manifolds. 
The configuration $\hat{\hat \delta}^f_r$ provides in this case a decomposition of the Hilbert module of  harmonic forms which depends continuously on $f.$ For a generic set of continuous functions $f,$ the Hilbert submodules  $\hat {\hat \delta}^f_r(z)$ have the von Neumann dimension equal to one.

To explain the construction of the configurations $\delta^f_r, \hat \delta ^f_r$ and $\hat{\hat\delta}^f_r$ some preliminaries presented in section 2.
are necessary. In section 3 one provides the definition of the configurations and a number of intermediate results and properties while in section 3 one indicates the way one verifies the statements in Theorem 1.1.  
\vskip .1in
The Author is thankful to the referee for pointing out a number of errors  and notational inconsistencies in a previous version of this paper.

\section {Preparatory material}

\subsection{Completion}\label{SS21}

Let $\mathbb C[t^{-1}, t]$ be the ring of Laurent polynomials, equivalently the group ring $\mathbb C [\mathbb Z]$ of the infinite cyclic group. This is an algebra with involution $\ast$ and trace $tr.$ If $a= \sum_{n\in \mathbb Z}  a_n t^n$ then:
\begin{equation*}
\begin{aligned}
 \ast (a):= a^\ast =& \sum_{n\in \mathbb Z} \overline a_n t^{-n} \\ tr (a)=&  a_0.
\end{aligned}
\end{equation*}
with $\overline a$ the complex conjugate of the complex number $a.$

The algebra $\mathbb C[\mathbb Z]$ can be considered as a sub algebra of the algebra of bounded linear operators on the separable Hilbert space $l_2(\mathbb Z),$ of square summable sequences $\{a_n, n\in \mathbb Z \mid \sum_{n\in \mathbb Z}  |a_n|^2 <\infty\} .$  The linear operator  defined by a Laurent polynomial is given by the multiplication of sequences in $l_2(\mathbb Z)$  with the Laurent polynomial regarded as a sequence with all but finitely many components equal to zero. 
One denotes by $\mathcal N$ the weak closure  of $\mathbb C[\mathbb Z]$ which is  a finite von Neumann algebra, with involution and trace extending the ones defined above, cf \cite{Lu}. 

This algebra $\mathcal N$ is referred to below as the von--Neumann completion of the group ring $\mathbb C(\mathbb Z)$ and is isomorphic to the familiar $L^\infty(\mathbb S^1)$ via Fourier series transform (which assigns to a complex valued function defined on $\mathbb S^1$ its Fourier series).   

Given a free $\mathbb C[t^{-1}, t]$--module  $M$ a {\it $\mathbb C[t^{-1}, t]$-valued  inner product}  is 
a map $\mu: M\times M\to \mathbb C^[t^{-1},t]$ which is: 
\begin{enumerate}
\item $\mathbb C[t^{-1},t]-$linear in the first variable,
\item {\it symmetric} in the sense that $\mu(x,y)= \mu(y,x)^\ast,$  $x,y\in M,$
\item {\it positive definite}  in the sense that satisfies 
\begin{enumerate}
\item $\mu(x,x)\in \mathbb C[t^{-1},t]_+$ with $\mathbb C [t^{-1},t]_+$ the set of elements of the form $a a^\ast$ and  
\item $\mu(x,x)=0$ iff $x=0,$ 
\end{enumerate}
and satisfies 
\item the map $M\to {\it Hom}_{\mathbb C^[t^{-1},t]} (M,\mathbb C^[t^{-1},t])$ defined by $\mu(y)(x)= \mu(x,y)$ is one to one. 
\end{enumerate}

Clearly $\mathbb C[t^{-1}, t]$-valued  inner products exist. Indeed, if $e^1, e^2, \cdots e^k$ is a base of $M$ then $$\mu (\sum a_ie^i, \sum b_j e^j):= \sum a_i (b_i)^\ast$$ provides  such inner product.

By completing the  $\mathbb C-$vector space $M$  w.r. to the  hermitian inner product $<x,y>:= tr(\mu(x,y))$ one obtains a Hilbert space $\overline M$ which is an $\mathcal N$-Hilbert module  cf \cite {Lu} isomeric to $l_2(\mathbb Z)^{\oplus k},$ $k$ the rank of $M.$
Two different $\mathbb C[t^{-1}, t]$-valued  inner products  $\mu_1$ and $\mu_2$ lead to the isometric Hilbert modules $\overline M_{\mu_1}$  and
$\overline M_{\mu_2}.$
This justifies dropping $\mu$  from notation. If one identifies $\mathcal N$ to $L^\infty (\mathbb S^1)$ and 
$l_2(\mathbb Z)^{\oplus k}$ to $L^2(\mathbb S^1)^{\oplus k}$  (by interpreting the sequence $\sum_{n\in \mathbb Z} a_n t^n$ as the complex valued function $\sum_{n\in \mathbb Z} a_n e^{i n\theta}$)  the $\mathcal N-$ module structure on $l_2(\mathbb Z)^{\oplus k}$  becomes the $L^\infty(\mathbb S^1)$-module structure on $(L^2(\mathbb S^1))^{\oplus k}$ and is given by the component-wise multiplication of  a $k-$tuple of $L^2$-functions, element in 
$(L^2(\mathbb S^1))^{\oplus k},$ by the $L^\infty$-function in $L^\infty (\mathbb S^1).$

If one has $N\subset M$ a free split submodule of the f.g free $\mathbb C[t^{-1},t]$-module $M$ and $\mu$ is an $\mathbb C[t^{-1}, t]$-valued inner product on $M$ then $\overline N_{\mu}$ is a closed Hilbert submodule of $\overline M_{\mu}.$ Moreover if $N'_i\subseteq N_i \subseteq M,$ $i= 1,2,\cdots $ is a collection of split submodules and $N_i/N'_i$ is a collection of free modules, quotient of submodules of $M,$ then 
one can canonically convert $N_i/ N'_i$  into closed Hilbert submodules of $\overline M$ simply by taking the closure of the kernel of the  projection $N_i\to N_i/ N_i'$ inside $N_i.$ The process of  passing from ($\mathbb C[t^{-1},t], M$) to  ($\mathcal N, \overline M$) is referred to below as {\it von Neumann completion} and was pioneered in \cite{Lu} for any group ring $\mathbb C[\Gamma] $ and f.g. projective $\mathbb C[\Gamma]$-module.    

\subsection{Configurations and the  collision topology on the space of configurations}

Let $X$ be a topological space and $N$ a positive integer.  
Denote by 
$$\mathcal C_N(X):=\{\delta: X\to \mathbb Z_{\geq 0} \mid \sum_{x\in X} \delta (x)= N\}$$
 the set of finite configurations of total cardinality $N.$  This set identifies to the space $X^N/\Sigma_N$ the 
 quotient space of the cartesian $N-$fold product of $X$ by the action of the permutation group $\Sigma _N$ of $N-$objects.
 
Let $V$ be a free f.g.module over the commutative unital ring $R$ or a finite type Hilbert module over a finite von Neumann algebra  $\mathcal N$  and let $\mathcal P(V)$ be the set of free split submodules  in the first case or of closed Hilbert submodules in the second. One generalizes  the set of configurations 
$\mathcal C_N(X)$ to the set $\mathcal C_V(X).$ The set $\mathcal C_V(X)$ consistis of maps $\hat {\hat \delta}: X\to \mathcal P(V)$ which satisfy
\begin{enumerate}
\item  $\supp (\hat{\hat \delta})= \{x\in X\mid \hat{\hat \delta}(x)\ne 0\}$ is finite 
\item   if $i(x): \hat{\hat \delta}(x)\to V$ denotes the inclusion  of  $\hat{\hat \delta}(x)$ in $V$ then the map $I$ 
$$I: = \sum_{x\in \supp (\hat{\hat \delta})} i(x): \bigoplus_{x\in \supp (\hat{\hat \delta})}  \hat{\hat \delta}(x)\to V$$ 
\end{enumerate}
is an isomorphism.
Denote by $$e: \mathcal C_V(X)\to \mathcal C_{\dim V}(X)$$ the map defined by $e(\hat{\hat \delta})(x):= \dim \hat{\hat\delta} (x).$

The set $\mathcal C_N(X)$ and the set ${\bf \mathcal C}_V (X)$ when $R= \mathbb C$ or when $V$ is an $\mathcal N-$ Hilbert module carry  natural topologies, referred to as the {\it collision topology}, which make
$e$ continuous.
One  way to describe these topologies is to specify for each $\delta$ or $\hat{\hat\delta}$  a system of {\it fundamental neighborhoods}.

If $\delta$ has as support  the set of points $\{x_1, x_2, \cdots ,x_k\},$ a fundamental neighborhood  $\mathcal U$ of $\delta$ is specified by a collection of $k$ disjoint open neighborhoods  $U_1, U_2,\cdots U_k$ of $x_1,\cdots  x_k,$ and consists of $\{\delta'\in \mathcal C_N(X)\mid
\sum_{x\in U_i} \delta'(x)=\delta(x_i)\}.$ 
Similarly, if $\hat{\hat \delta}$ has as support  the set of points $\{x_1, x_2, \cdots  x_k\}$ with $\hat{\hat\delta}(x_i)= V_i\subseteq V\},$ a fundamental neighborhood $\mathcal U$  of $\hat{\hat \delta}$ is specified by a collection of  disjoint open neighborhoods  $U_1, U_2,\cdots U_k$ of $x_1,\cdots  x_k,$ and open neighborhoods $O_1, O_2, \cdots, O_k$ of $V_1, V_2, \cdots, V_k$ in $G_{\dim V_i} (V)$ 
and  consists of 
$$\{\hat{\hat \delta}'\in \mathcal C_V(X) \mid  \sum_ {x\in U_i} \hat{\hat \delta}' (x) \in O_i\}.$$
Here $G_k(V)$ denotes the Grassmanian of $k-$dimensional subspaces  of $V$  \footnote {When $V$ is an $\mathcal N$--Hilbert module $G_k(V)$ can be identified to the set of $\mathcal N$--linear  self adjoint projectors whose von Neumann trace is equal to $k$ which inherits the topology induced by the norm of bounded operators in the Hilbert space $V.$ }.

Clearly $e$ is continuous and surjective  with fiber above $\delta,$ the subset of  $G_{n_1}(V)\times G_{n_2}(V) \cdots \times G_{n_k}(V)$ consisting  of $(V'_1, V'_2,\cdots V'_k), V'_i\in G_{n_i}(V)$ 
where $n_i= \dim V_i.$  

Note that: 
\begin{enumerate}
\item $\mathcal C_N(X)= X^N/\Sigma_N$ is the $N$--fold symmetric  product of $X$ and if $X$ is  a metric space with distance $D$ then
the collision topology is  the topology defined by  the induced distance $\underline D$ on  $X^N/\Sigma_N.$
\item If $X= \mathbb R^2= \mathbb C$ then $\mathcal C_N(X)$ identifies to the set of monic polynomials with complex coefficients. To the configuration $\delta$ whose support consists of the points $z_1, z_2, \cdots z_k$ with  $\delta(z_i)= n_i$ one associates the monic polynomial  $P^f(z)= \prod _i (z- z_i)^{n_i}. $ Then $\mathcal C_N(X)$ identifies to $\mathbb C^N$ as metric spaces. 
\item Similarly, if $X= \mathbb C^\ast=  \mathbb C\setminus 0$ then $\mathcal C_N(X)$ identifies to the set of  monic polynomials of degree $N$ with non vanishing  free coefficient, hence with $\mathbb C^{N-1}\times \mathbb C^\ast$ where $\mathbb C^\ast= \mathbb C\setminus 0.$
\end{enumerate}

In this paper we will consider as intermediate step a slightly more general type of configurations involving  {\it quotients of  free split submodules} of a free $R-$module or {\it quotients  of closed Hilbert submodules}; actually only the case $R= \kappa[t^{-1},t]$  will be involved.

Let $M$ be a f.g. free $\kappa[t^{-1},t]$--module. Denote by $\tilde{\mathcal S}(M)$ the set of pairs $(L\supset L') $  each  pair with $L, L'$  split submodules of $M.$ Since $M$ is f.g. and free so are $L$ and $L'$ and $L/L'.$

A finite collection of pairs $(L_r\supset L'_r)\in \tilde {\mathcal S}(M), \  r=1,2,\cdots k,$   is called {\it well separated } if for any right inverses  $i_r: L_r/L'_r \to L_r\subseteq M$ of the projections $L_r\to L_r/L'_r$ the sum of linear map $$\sum_{1\leq r
\leq k} i_r : \bigoplus L_r/ L'_r \to M$$ is injective.

For a map $\hat \delta: X\to \tilde {\mathcal S} (M)$ denote by $\supp (\hat\delta)$ the set  $$\supp (\hat\delta):= \{x\in X \mid \hat\delta(x)= (L(x), L'(x)), \ L(x)\ne 0\}.$$
A finite configuration of quotient of  split submodules  of $M$ is a map  $\hat\delta: X \to \tilde {\mathcal S} (M)$ which satisfies:
\begin{enumerate}
\item $\supp (\hat \delta)$  is finite,
\item The collection of pairs $(L(x) \supset L'(x))$  is well separated. 
\item 
For any right inverses $i(x)'$s the linear map 
 $$\sum_{x\in \supp (\hat \delta)} i(x)  :  \bigoplus _{x\in \supp (\hat \delta)}L(x)/L'(x) \to M$$ is an isomorphism. 
\end{enumerate} 
When $R= \mathbb R$ or $\mathbb C,$ in the presence of an scalar product  (Hilbert space structure) on $V,$ one can canonically pass from  a map as above $\hat \delta$ to a map $\hat{\hat \delta}$ by replacing the pair $L(x)\supset L'(x)$ by the orthogonal complement of $L'(x)$ in $L(x)$. This is also the  case when $V$ is a $\mathcal N-$Hilbert module. 

In view of the subsection \ref{SS21}, when $\kappa= \mathbb C,$   the choice of a $\mathbb C[t^{-1},t]$-valued  inner product on $M$ provides a hermitian inner product in $M$ as explained in the previous subsection and the von Newman completion converts  any configuration $\hat \delta$  into a configuration $\hat{\hat \delta} $ of closed Hilbert submodules of $\overline M.$

\subsection {Preliminary on compact ANR's and tame maps} \label {SS23}\

{\it Tame maps:}  For a continuous map $f:X\to Y$ between two topological Hausdorff spaces a {\it regular value} is a point $y\in Y$ for which there exists a neighborhood $U$ of $y$ s.t. for any $y'\in U$ the inclusion $f^{-1}(y')\subset f^{-1}(U)$ is a homotopy equivalence. The values $y$ which are not regular are called {\it critical} and a map is {\it tame} if the set of critical values $Cr(f)\subset Y$ is discrete.  In case $Y$ is a metric space with distance $d,$ in particular $Y= \mathbb R$ or $\mathbb S^1,$ for a  map $f$ one can introduce $\epsilon(f): = \inf_{y_1,y_2\in Cr(f), y_1\ne y_2)}  d(y_1, y_2).$  
If  $Y= \mathbb R$ or $\mathbb S^1,$ $X$ is compact and $f$ is tame then $\epsilon(f)>0$ and if $f$ is not tame then $\epsilon (f)=0.$

{\it ANR's and Hilbert cube manifolds}: One  denotes by $[0,1]^\infty$ the countable product of the compact interval $[0,1]$ and call it the {\it Hilbert cube}.  A second countable  Hausdorff space is a Hilbert cube manifold if  is locally homeomorphic to $[0,1]^\infty.$  In view of fundamental results on the topology of Hilbert cube manifolds due to Edwards, Chapman, West etc, cf  \cite {CH},   
{\it   a locally compact space $X$ is an ANR iff the product with $[0,1]^\infty$ is a Hilbert cube manifold.}  

 We are concerned in this paper with compact ANRs.  A space $X$ is a compact ANR iff stably homeomorphic to a finite simplicial complex, i.e. iff there exists a simplicial complex $K$ such that 
$X\times [0,1]^\infty$ is homeomorphic to $K\times [0,1]^\infty.$  Recall also that two compact Hilbert cube manifolds are homeomorphic iff they are homotopy equivalent cf \cite{CH} however not any homotopy equivalence is homotopic to a homeomorphism \footnote {but only the simple homotopy equivalences cf \cite {CH}}.
The p.l. maps from a simplicial complex $K$ into $\mathbb R$ or $\mathbb S^1$ are dense in  the space of continuous maps with  compact open topology.  Any  p.l map is tame if $K$ is finite. For a compact Hilbert cube manifold the tame maps are also dense in the space of continuous maps (with compact open topology).  

We will  use these results on compact Hilbert cube manifolds  to establish results about general compact ANRs  by verifying first their validity for finite simplicial complexes.

\section {The configurations $\delta^f_r$,  $\hat \delta^f_r$ and $\hat{\hat {\delta}}^f_r.$}.

Let $X$ be a compact ANR and $f:X\to \mathbb S^1$ be a continuous  map. An infinite cyclic cover of $f$ is provided by the commutative diagram
\begin{equation}
\xymatrix{ &\mathbb R \ar[r]^p & \mathbb S^1\\
&\tilde X\ar[u]^{\tilde f}\ar[r]^\pi &X\ar[u]^f }
\end{equation}
and the free action $ \mu:\mathbb Z\times \tilde X\to\tilde X$ with $\pi(\mu(n,x))= \pi(x)$ inducing an homeomorphism $\pi: \tilde X/\mathbb Z\to X$ and $\tilde f(\mu(n,x))= \tilde f(x)+2\pi n.$ 
Denote by $\tau: \tilde X\to \tilde X$  the restriction of $\mu$ to $1\times \tilde X.$
\vskip .1in
Fix $\kappa$ a field. To ease the writing for a space $Y$ we abbreviate $H_r(Y;\kappa)$ by $H_r(Y).$
The homeomorphism $\tau$ induces the isomorphism $t_r: H_r(\tilde X)\to H_r(\tilde X)$  which defines a structure of 
$\kappa[t^{-1},t]$-module on the $\kappa-$vector space $H_r(\tilde X).$ The isomorphism $t_r$ represents the multiplication by $t\in \kappa[t^{-1},t].$ Note that $ H_r(\tilde X)$ is a f.g.  $\kappa[t^{-1}, t]$-module  and that $\kappa[t^{-1},t]$ is a principal ideal domain therefore the torsion submodule $T(H_r(\tilde X))$ is a finite dimensional $\kappa-$vector space,  $H^N_r(X,\xi):= H_r(\tilde X)/T( H_r(\tilde X)) $ is a f.g. free $\kappa[t^{-1},t]-$module and $H_r(\tilde X)$ is isomorphic to $H^N_r(X,\xi)\oplus T(H_r(\tilde X).$
\vskip .1in

Denote by $\tilde X_a= \tilde f^{-1}((-\infty,a])$, $\tilde X^b= \tilde f^{-1}([b,\infty)). $
Following \cite {Bu} one introduces the following notions:

\begin{itemize}
\item $\mathbb I_a(r) = \img(H_r(\tilde X_a)\to H_r(\tilde X)),$  $\mathbb I^b(r)= \img(H_r(\tilde X^b)\to H_r(\tilde X)),$
\item $\mathbb I_{-\infty}(r): =\cap_{a\in \mathbb R} \mathbb I_a(r),$ \quad 
$\mathbb I^{\infty}(r): =\cap_{b\in \mathbb R} \mathbb I^b(r),$ 
\item $\mathbb F_r(a,b)= \mathbb I_a(r) \cap \mathbb I^b(r),$ 
\item $\mathbb F_r(-\infty,b)= \mathbb I_{-\infty}(r) \cap \mathbb I^b(r),$  
\quad $\mathbb F_r(a,\infty)= \mathbb I_a(r) \cap \mathbb I^\infty (r).$
\end{itemize}

With this notation one observes  as in \cite{Bu}  that: 
\begin{obs}\label{O31}\ 
\begin{enumerate}
\item $t_r: \mathbb I_a(r)\to \mathbb I_{a+2\pi}(r)$ and $t_r:  \mathbb I^b(r)\to \mathbb I^{b+2\pi}(r)$ are isomorphisms 
and therefore: 
$t_r:\mathbb F_r(a,b) \to \mathbb F_r(a+2\pi, b+2\pi) $ is an isomorphism. 
Then 
both $\mathbb I_{-\infty}(r)$ and $\mathbb I^\infty(r)$ are $\kappa[t^{-1},t]$-submodules.  
\item For $a'\leq a$ and $b\leq b'$ one has:

$\mathbb F_r(a',b')\subseteq \mathbb F_r( a,b),$ 

$\mathbb F_r(-\infty ,b')\subseteq \mathbb F_r( a,b)$ and
$\mathbb F_r(a',\infty)\subseteq \mathbb F_r( a,b).$
\item $\cup_{a\in \mathbb R} \mathbb I_a(r)= \cup_{b\in \mathbb R}\mathbb I^b(r)= H_r(\tilde X).$
\end{enumerate}
\end{obs}

\begin{proposition}\label {P32}\ 

1. The dimension of $\mathbb F_r(a,b)$ is finite.

2. $\mathbb I_{-\infty}(r) = \mathbb I^{\infty}(r)= T(H_r(\tilde X)).$
\end{proposition} 

\begin{proof}\

Item 1. is verified in 
\cite{Bu} based on Meyer--Vietoris sequence and on the observations that $\tilde X$ is a locally compact  ANR and  $\tilde f$ is a proper map.  

Item 2.:  If $x\in T(H_r(\tilde X)) $ then there exists  an integer $k\in \mathbb Z$ and  a polynomial $P(t)= \alpha_n t^n +\alpha_{n-1} t^{n-1} \cdots \alpha_1 t + \alpha_0,$ $\alpha_i\in \kappa,$ $\alpha_0\ne 0$ such that $P(t)\ t^k x=0.$  

Let $y=t^k x.$  By Observation \ref{O31} item 3. one has  $y\in \mathbb I^b$ for some $b \in \mathbb R.$ 
Since  $P(t) y=0$  one concludes that  
 $y= - (\alpha_n/\alpha_0) t^{n-1} \cdots - (\alpha_1/\alpha_0) t y$ and then by Observation \ref{O31} item 1. one has $y\in \mathbb I^{b+2\pi}.$  
 
 Repeating the same argument  one concludes that $y\in I^{b+2\pi k}$ for any $k,$ hence $y \in \mathbb I^\infty.$ Since $x= t^{-k} y,$  by Observation \ref{O31} item 1.  one has  $x\in \mathbb I^\infty.$ Hence $T(H_r(\tilde X))\subseteq I^\infty.$

Let $x\in \mathbb I^\infty.$  By Observation \ref{O31} item 3. one has  $x\in \mathbb I_a$ for some $a\in \mathbb R$ and if $x\in \mathbb I^\infty,$  by Observation \ref {O31} item 1. , all $x, t^{-1}x, t^{-2} x, \cdots t^{-k}x,\cdots \in \mathbb I_a\cap \mathbb I^\infty.$  Since by (1.) 
 the dimension of $\mathbb I_a\cap \mathbb I^\infty$ is finite, there exists $\alpha_{i_1}, \cdots \alpha_{i_k}$ such that $(\alpha_{i_1} t^{-i_1} + \cdots \alpha_{i_k} t^{-i_k}) x=0.$  This makes $x\in T(H_r(\tilde X))$. Hence $\mathbb I^\infty\subseteq T(H_r(\tilde X)).$  
Therefore $\mathbb I^\infty = T(H_r(\tilde X)).$ 
 
 By a similar arguments one concludes that $T(H_r(\tilde  X))= \mathbb I_{-\infty}.$

\end{proof}

\vskip .1in 
It is convenient to consider a  weaker concept of critical values relative to homology with coefficients in the fixed field $\kappa.$

\begin{definition}\label {D33}\
\begin{enumerate}
\item  A real number $c$  is  a {\bf sub level homologically critical values} if for any $\epsilon >0 $
the inclusion $\mathbb I_{c-\epsilon}(r)\subset \mathbb I_{c+\epsilon}(r)$  is strict (not equality). Denote by $CR_-(\tilde f)$ the set of such critical values.
\item  A real number $c$  is  a {\bf over level homologically critical values} if for any $\epsilon >0 $
the inclusion $\mathbb I^{c+\epsilon}(r)\subset \mathbb I^{c-\epsilon}(r)$  is strict (not equality). Denote by $CR^+(\tilde f)$ the set of such critical values.
\item A real number $c$  is  a {\bf homologically critical values} if it belongs to 
$CR(\tilde f):= CR_-(\tilde f)\cup CR^+(\tilde f)$
\end{enumerate}
\end{definition}
Observation  \ref{O31} can be completed with the following observation.

\begin{obs}\label {O32}\ 
Suppose $\tilde f:\tilde X\to \mathbb R$ is an infinite cyclic cover of $f:X\to \mathbb S^1,$ $X$ a compact ANR.
\begin{enumerate}
\item $CR_-(\tilde f)$, $CR^+(\tilde f)$ and then $CR(\tilde f)$ are discrete sets. Moreover there exists $\tilde \epsilon (f) >0$ such that $\tilde \epsilon(f) <|c'-c''| \ $  for $c',c''\in CR(\tilde f), c'\ne c''.$ 
\item  The map  $f$ is tame iff $\tilde f$ is tame.
\item  $CR(\tilde f)\subseteq Cr(\tilde f)$ and $\tilde \epsilon(f) \geq \epsilon (f).$
\end{enumerate}
\end{obs}
Item 1. follows from Observation \ref{O31} and Proposition \ref{P32}, 
while  items 2. and 3. are  straightforward consequences of definitions. 
\vskip .2in
{\it Boxes;} For $a'<a, b< b'$ one considers  the domain $B=(a',a]\times [b,b')$ (see the Figure 1) and call it a finite box.  An infinite box is of the form $(-\infty,a]\times [b,b')$ or $(-\infty,a]\times [b,\infty)$ or $(a',a]\times [b,\infty).$

\vskip .2in

\begin{figure}
\begin{center}
\begin{tikzpicture}
\draw [<->]  (0,3) -- (0,0) -- (3,0);
\node at (-0.5,2.8) {y-axis};
\node at (2.9,-0.2) {x-axis};
\node at (1,0.3) {(a',b)};
\node at (2.9,0.3) {(a,b)};
\node at (3,2.3) {(a,b')};
\node at (1,2.3) {(a',b')};
\draw [<->]  (0,-3) -- (0,0) -- (-3,0);
\draw [thick] (1,0.5) -- (2.9,0.5);
\draw [thick] (2.9,0.5) -- (2.9,2);
\draw [dotted] (2.9,2) -- (1,2);
\draw [dotted] (1,2) -- (1,0.5);
\draw (0,0) -- (2.5,2.5);
\draw (0,0) -- (-2.5,-2.5);
\end{tikzpicture}
\caption {The {\it box} $ B : =(a',a]\times [b,b')\subset \mathbb R^2$  }
\end{center}
\end{figure}
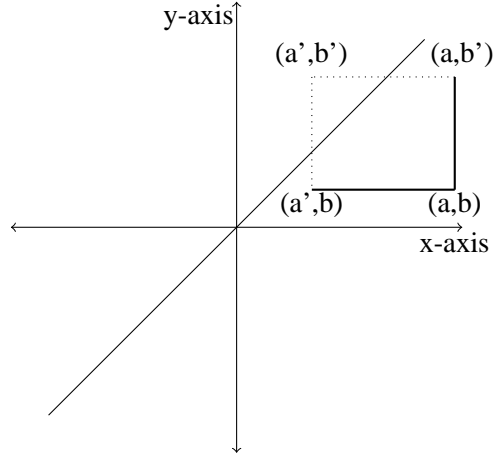
For a box $B$  denote by: 
\begin{equation}
\begin{aligned} 
\mathbb F'_r(B):= \mathbb F_r(a',b)& + \mathbb F_r(a,b') \subseteq \mathbb F_r(a,b)\\
 \mathbb F_r(B):= &\mathbb F_r(a,b)/ \mathbb F'_r( B)
 \end{aligned}
 \end{equation}
and let $\pi^B_{ab,r}: \mathbb F_r(a,b)\to \mathbb F_r(B)$ the projection on quotient space. 

\newcommand{\mypicture}[1][ ]{
\begin{tikzpicture} [scale=1]
\draw [line width=0.10cm] (0,0) -- (5,0);
\draw [line width=0.10cm] (5,0) -- (5,3);
\draw [dashed, ultra thick] (0,3) -- (5,3);
\draw [dashed, ultra thick] (0,0) -- (0,3);
\draw [line width=0.10cm] (3,0) -- (3,3);
\node at (1.5,1.5) {B1};
\node at (4,1.5) {B2};
\node at (2.5, -0.5){Figure 2};
\end{tikzpicture}
\hskip  .5in
\begin{tikzpicture}[ ] [scale=0.8]
\draw [line width=0.10cm] (7,0) -- (12,0);
\draw [line width=0.10cm] (12,0) -- (12,3);
\draw [dashed, ultra thick] (7,0) -- (7,3);
\draw [dashed, ultra thick] (7,3) -- (12,3);
\draw [line width=0.10cm] (7,1) -- (12,1);
\node at (9.5,2) {B1};
\node at (9.5,0.5) {B2};
\node at(9.5, -0.5) {Figure 3};
\end{tikzpicture}}
\vskip .1in

\newcommand{\mynewpicture}[1][ ]{
\begin{tikzpicture} [scale=1]
\draw [dashed, ultra thick] (0,0) -- (0,3);
\draw [line width=0.10cm] (5,0) -- (5,3);
\draw [line width=0.10cm] (2,1) -- (2,3);
\draw [line width=0.10cm] (0,1) -- (2,1);
\draw [line width=0.10cm] (0,0) -- (5,0);
\draw [dashed, ultra thick] (0,3) -- (5,3);
\node at (1,2) {B'};
\node at (3.5,0.8) {B''};
\node at (2.5, -0.5) {Figure 4};
\end{tikzpicture}
\hskip .5in
\begin{tikzpicture} [scale=1]
\draw [dashed, ultra thick] (0,0) -- (0,3);
\draw [line width=0.10cm] (5,0) -- (5,3);
\draw [dashed, ultra thick] (3,1.5) -- (5,1.5);
\draw [dashed, ultra thick] (3,0) -- (3,1.5);
\draw [line width=0.10cm] (0,0) -- (5,0);
\draw [dashed, ultra thick] (0,3) -- (5,3);
\node at (1,2) {B''};
\node at (3.5, 0.8) {B'};
\node at (2.5, -0.5) {Figure 5};
\end{tikzpicture}}

\vskip .1in

In view of the definition of $\mathbb F_r(a,b)$ and of $\mathbb F_r(B)$  it is straightforward to observe (as in \cite{Bu}) 
that the following statements are  true for either finite or infinite boxes. The linear maps involved are ultimately induced by inclusions $\mathbb F_r(a',b')\subseteq \mathbb F_r(a,b)$ for $a'\leq a, b\leq b'$  (cf  Observation\ref{O31} item 2) and  by '' passings to quotient spaces''.  

\begin{obs}\label{O35}\
\begin{enumerate}
\item  If $a''<a'<a ,$  $b< b''$ (possibly $a''=-\infty, b''=\infty$) and $B_1,$ $B_2$ and $B$  are the boxes ,$B_1= (a'',a']\times [b,b'),$ $B_2= (a',a]\times [b,b'')$
and $B= (a'',a]\times [b,b') $ (see Figure 2)
then the inclusions $B_1\subset B$ and $B_2\subset B$ induce the linear maps 
${i^B_{B_1,r}}: \mathbb F_r(B_1)\to  \mathbb F_r(B)$ and $\pi^{B_2}_{B,r}:  \mathbb F_r(B)\to  \mathbb F_r(B_2)$  such that the  sequence 
$$\xymatrix{0\ar [r]& \mathbb F_r(B_1)\ar[r]^{i^B_{B_1,r}} & \mathbb F_r(B)\ar[r]^{\pi^{B_2}_{B,r}}& \mathbb F_r(B_2)\ar[r]& 0}.$$ 
is exact.
\item If $a'<a,$ $b<b'< b''$ (possibly $a'=-\infty, b''=+\infty$) and $B_1,$ $B_2$ and $B$  are the boxes, $B_1= (a',a]\times [b',b''),$ $B_2= (a',a]\times [b',b')$
and $B= (a',a]\times [b,b'') $ (see Figure 3) 
then the inclusions $B_1\subset B$ and $B_2\subset B$ induce the linear maps 
$i^B_{B_1,r}:  \mathbb F_r(B_1)\to  \mathbb F_r(B)$ and $\pi^{B_2}_{B,r}:  \mathbb F_r(B)\to  \mathbb F_r(B_2)$  such that the  sequence $$\xymatrix{0\ar [r]& \mathbb F_r(B_1)\ar[r]^{i^B_{B_1,r}} & \mathbb F_r(B)\ar[r]^{\pi^{B_2}_{B,r}} & \mathbb F_r(B_2)\ar[r]& 0}$$ is exact.

\mypicture {}

\item If $B'$and $B''$ are two boxes with $B'\subseteq B''$ and  $B'$ is located in the upper left corner of $B''$ (see  Figure 4) then the inclusion of boxes induces the canonical injective linear maps $i^{B''}_{B',r}: \mathbb F_r(B')\to \mathbb F_r(B'').$ 
 If $B'$and $B''$ are two boxes with $B'\subseteq B''$ and  $B'$ is located in the lower right  corner of $B''$ (see Figure 5) then the inclusion of boxes induces the canonical surjective  linear maps $\pi^{B''}_{B',r}: \mathbb F_r(B')\to \mathbb F_r(B'').$ 
\vskip .2in
\mynewpicture { }
 \item If $B$ is a finite disjoint union of boxes $B=\sqcup B_i$ then $\mathbb F_r(B)$ is isomorphic to $\oplus _i \mathbb F_r(B_i);$
the isomorphism is not canonical. 
\end{enumerate}
\end{obs}
\vskip .2in

For $\epsilon >0$ one denotes by $B(a,b; \epsilon): = (a-\epsilon,a]\times [b, b+\epsilon)$  and then for $\epsilon' >\epsilon$ one 
has the surjective linear maps $ \mathbb F_r(B(a,b; \epsilon'))\to \mathbb F_r(B(a,b; \epsilon)).$   
In view of Proposition  \ref{P32} item 1. the  limit  
$$\hat \delta^{\tilde f}_r(a,b):= \varinjlim _{\epsilon\to 0} \mathbb F_r(B(a,b; \epsilon))$$ exists and one defines
$$\delta^{\tilde f}_r(a,b):= \dim (\hat \delta^{\tilde f}_r(a,b)).$$
In view of Observation \ref{O32} item 1. the limit stabilizes and for $\epsilon$ small enough and one has  

\begin{equation}\label {E2}
\hat \delta^{\tilde f}_r(a,b)= \mathbb F_r(B(a,b; \epsilon)).
\end{equation} 

It is useful to regard $\hat \delta^{\tilde f}_r(a,b)$ not only as a vector space but as  the quotient of subspaces of $H_r(\tilde X)$ 
$$\hat \delta^{\tilde f}_r(a,b)= \mathbb F_r(a,b) / \mathbb F'_r(a,b),$$  
with $$\mathbb F'_r(a,b):=\mathbb F'_r(B(a,b; \epsilon)) = \mathbb F_r(a-\epsilon, b) +\mathbb F_r(a, b+\epsilon)\subseteq \mathbb F_r(a,b)\subseteq H_r(\tilde X)$$ for  
$\epsilon$ small enough 
\footnote {The right side of the equality above is 
independent of  $\epsilon$ if  $\epsilon$ is small enough}.  Note that if at least one of $a,b$ are not homologically critical values then $\mathbb F_r(a,b)= \mathbb F'_r(a,b)$ and if both $a,b$ are homologically critical values then  the stabilization $\mathbb F'_r(a,b)= \mathbb F'_r(B(a,b;\epsilon))$ holds for 
any $\epsilon$ s.t. $0 < \epsilon <\tilde \epsilon (f).$
\vskip .1in
Define 
$$\supp (\delta^{\tilde f}_r)=\supp (\hat \delta^{\tilde f}_r):  =\{(a,b)\mid \delta^{\tilde f}_r(a,b)\ne 0\}.$$

As a consequence of equality (\ref{E2}) and of Observation \ref{O35}
one has 
\begin{proposition} \label{P36}\

\begin{enumerate} 
\item $\delta^{\tilde f}_r(a,b)\ne 0$ implies both $a,b \in CR(\tilde f).$
\item For any finite or infinite box $B$ the set $\supp (\delta^{\tilde f}_r) \cap B$ is  finite.
\item For any finite or infinite box $$\sum_{(a,b)\in \supp (\delta^{\tilde f}_r)}  {\delta}^{\tilde f}_r(a,b) = \dim \mathbb F_r(B).$$
\end{enumerate}
\end{proposition}

\begin{proof}\ 

Item 1. follows from Definition \ref{D33} and Observation \ref{O35} item 1. .

Item 2.: If  $B\cap \supp (\delta^{\tilde f}_r)$ is an infinite set  then in view of Observation \ref{O35} item 4. and of the equality (\ref{E2}) $\dim \mathbb F_r(B)$ is infinite. This is not possible since this dimension is  smaller than $\dim \mathbb F_r(a,b)$ with $(a,b)$  the right bottom corner of $B,$ which by Proposition \ref{P32} is finite. 

Item 3. : By Observation \ref{O35} item 3. and the equality (\ref{E2}),  if $B\cap \supp (\delta^{\tilde f}_r)=\emptyset$  then $\mathbb F_r(B)= 0,$ and if 
$B\cap \supp (\delta^{\tilde f}_r)=(a',b')$  then $\mathbb F_r(B)= \hat \delta^{\tilde f}_r(a',b').$ 
The statement  in full generality follows by observing that any box, finite or infinite, can be divided in a finite collection  of disjoint boxes 
$B= \sqcup B_{i,j}, 1\leq i\leq m, 1\leq j\leq n,$ s.t $B_{i,j}$ and $B_{i+1,j}$ are in the position indicated in Figure 2, $B_{i,j}$  and $B_{i,j+1}$ in the position indicated in Figure 3 and each $B_{i,j}$  contains at most one element of $B\cap \supp (\delta^{\tilde f})_r.$  The result is established by induction on $m$ and $n$   by applying Observation (\ref {O35}) item 1. and item 2. and the particular cases mentioned above.
\end{proof}. 
\vskip .2in

The above proposition can be strengthen in the following way. 
For each $(a,b)$ consider the surjective map  $$\pi_r(a,b): \mathbb F_r(a,b)\to \hat \delta^{\tilde f}_r(a,b)$$ and call {\it splitting} a right invers of $\pi_r(a,b),$  $$s_r(a,b): \hat \delta^{\tilde f}_r(a,b)\to \mathbb F_r(a,b).$$ 

For each box $B= (\alpha', \alpha]\times [\beta,\beta')$ with $a\in(\alpha',\alpha],\ b\in [\beta,\beta')$ 
denote by $$i_r^B(a,b) : \hat \delta^{\tilde f}_r(a,b)\to \mathbb F_r(B) \ \rm{and} \  \ i_r(a,b):\hat \delta^{\tilde f}_r(a,b)\to \mathbb H_r(\tilde X) $$  the compositions  
$$\xymatrix {\hat \delta^f_r(a,b)\ar[r]^{s_r(a,b)} &\mathbb F_r(a,b)\ar[r]^{\subseteq} &\mathbb F_r(\alpha, \beta)\ar [r]^{\pi^B_{\alpha \beta ,r}}&\mathbb F_r(B)}$$ and 
$$\xymatrix {\hat \delta^f_r(a,b)\ar[r]^{s_r(a,b)} &\mathbb F_r(a,b)\ar[r]^{\subseteq} &H_r(\tilde X)}.$$
\vskip .1in 

The following diagram reviews for the reader the linear maps considered so far 

\vskip .2in 

\begin{equation}\label {D111}
\xymatrix{ H_r(\tilde X) & \mathbb F_r(a,b)\ar[d]_{\pi^B_{ab,r}}\ar[l]_\supseteq \ar[r]_{\pi_r(a,b)} &\hat \delta^{\tilde f}_r(a,b)\ar@/_1pc/[l]_{s_r(a,b)}\ar@/_2 pc/[ll]_{i_r(a,b)}   \ar[ld]^{i^B_r(a,b)}\\
\mathbb F_r(B_1)\ar[r]^{i^B_{B',r}}&\mathbb F_r(B)\ar[r]_{\pi^{B_2}_{B,r}}&\mathbb F_r(B_2).}
\end{equation}
Observe that if $B=B_1\sqcup B_2$ as in Figure 2 or Figure 3, in view of Observation \ref{O35},  one has 
\begin{obs}\label {O36}\
\begin{enumerate}
\item If $(a,b)\in B_2$  then $\pi^{B_2}_{B,r}\cdot  i^B_r(a,b)$ is injective. 
\item If $(a,b)\in B_1$  then $\pi^{B_2}_{B,r}\cdot  i^B_r(a,b)$ is zero.
\end{enumerate}
 \end{obs}
\vskip .1in 

Choose splittings $\{ s_r(a,b)\mid (a,b)\in \supp (\delta^{\tilde f}_r)\},$ 
and consider  the sum of $i_r(a,b)'$s for $(a,b)\in \supp (\delta^{\tilde f}_r).$
$$ I_r= \sum _{(a,b)\in \supp (\delta^{\tilde f}_r)} i_r(a,b): \bigoplus_{(a,b)\in \supp (\delta^{\tilde f}_r)} \hat\delta^f_r(a,b) \to H_r(\tilde X).$$ 
 and for a finite or infinite box $B$  
the sum 
$$I^B_r= \sum_{(a,b)\in \supp (\delta^{\tilde f}_r) \cap B} i^B_r(a,b): \bigoplus_{(a,b)\in \supp (\delta^{\tilde f}_r) \cap B} \hat\delta^f_r(a,b) \to \mathbb F_r(B).$$ 

For $\Sigma\subseteq \supp (\delta^f_r)$ denote by $I_r(K)$ the restriction of $I_r$ to $\bigoplus_{(a,b)\in \Sigma} \hat\delta^f_r(a,b) $ and for $\Sigma \subseteq  \supp (\delta^f_r) \cap B$  denote by  $I^B_r(K)$ the restriction of $I^B_r$ to $\bigoplus_{(a,b)\in \Sigma} \hat \delta^f_r(a,b).$
Note that 
\begin{obs} \label {O37}\ 

For $B=B_1\sqcup B_2$ as in Figures 2 or Figure 3 and $\Sigma\subseteq \supp\ \delta ^{\tilde f}_r $ with $\Sigma= \Sigma_1 \sqcup \Sigma _2$, $\Sigma_1 \subseteq B_1, \Sigma_2\subseteq B_2$ the  diagram 
$$\xymatrix{ \mathbb F_r(B_1) \ar[r] & \mathbb F_r(B) \ar[r] &\mathbb F_r(B_2)\\
\bigoplus _{(a,b)\in \Sigma_1}\hat \delta^{\tilde f}_r(a,b) \ar[u]_{I^{B_1}_r(\Sigma_1)} \ar[r]&\bigoplus_{(a,b)\in \Sigma} \hat \delta^{\tilde f}_r(a,b) \ar[u]_{I^{B}_r(\Sigma)}\ar[r] &\bigoplus_{(a,b)\in \Sigma_2} \hat \delta^{\tilde f}_r(a,b) \ar[u]_{I^{B_2}_r(\Sigma_2)}}
$$
is commutative.  In particular if $I^{B_1}_r(\Sigma_1) $ and $I^{B_2}_r(\Sigma_2)$ are injective then so is $I^B_r(\Sigma).$
\end{obs}
\vskip .2in 

\begin{proposition}\label{P37}\
\begin{enumerate}
\item For any $\Sigma$ as above the linear maps $I_r (\Sigma),$   and $I^B_r(\Sigma)$ are  injective. 
\item $I^B_r$ is an isomorphism.
\item If $\pi_r$ is the projection $\pi_r: H_r(\tilde X)\to  H_r(\tilde X)/ (\mathbb I_{-\infty}(r) + \mathbb I^\infty(r))= H_r(\tilde X)/ T H_r(\tilde X)$  then 
$\pi_r\cdot I_r$ is an isomorphism. 
\end{enumerate}
\end{proposition}
\begin{proof}\

Item 1:  If cardinality of $\Sigma$ is one then the statement is verified by Observation \ref{O36}.  If all  points  of $\Sigma$ have the same first component the statement follows by induction on the cardinality of $\Sigma$  using Observation \ref{O37}. Precisely one decomposes $B$ as  in Figure 2,  $B=B_1\sqcup B_2$ with $B_1$ containing $k-1$ elements and $B_2$ containing one element. Here $k$ is the cardinality of 
$\Sigma.$ The statement, being true for $\Sigma \cap B_2$ and by induction hypothesis for  $\Sigma \cap B_1,$ in view of Observation \ref{O37}, holds true for $\Sigma.$

In general one decomposes the set $\Sigma$ as $\Sigma= \Sigma_1\sqcup \Sigma_2\cdots \sqcup \Sigma_k$ with the properties that all elements of 
$\Sigma _i$ have the same first component, say $a_i,$ with $a_1 < a_2 \cdots  < a_k.$ One proceeds by induction on $k$ by decomposing $B$ as  in Figure 3,  $B= B_1\sqcup B_2$with $B_2$ containing the set $\Sigma _1$ and $B_1$ the remaining elements.  By the previous step the statement is true for $\Sigma _1$ and by induction hypothesis for $\Sigma \cup B_1.$  In view of Observation \ref{O37} the statement holds for $\Sigma.$ 

Item 2.: Injectivity is true  by item 1..  The surjectivity 
 follows from equality of the dimension of the source  and of the target which follows  from Proposition \ref{P36} .

 Item 3.:  In view of Observation \ref{O31} item 3. one has: $$\varinjlim _{k\to \infty}  \mathbb F_r(a+k,b-k) = H_r(\tilde X)$$ 
and $$\varinjlim_{k\to \infty} \mathbb F_r((-\infty,a+k]\times [b-k, \infty)):= H_r(\tilde X)/ ( \mathbb I_{-\infty} +\mathbb I_\infty)= H_r(\tilde X) / TH_r(\tilde X) .$$
 For a fix $(a,b)$  denote by $B_k:=  (-\infty,a+k]\times [b-k, \infty)$  and regard $\mathbb R^2:= \cup _k B_k.$
The statement follows  by observing that $\pi_r\cdot I_r= \varinjlim _{\{k\to \infty\}}  I^{B_k}_r.$
\end{proof}
\vskip .2in 

{\it The configurations $\delta^f_r$ and $\hat \delta^f_r$}.  Let $\omega: \mathbb Z\times \mathbb R^2 \to \mathbb R^2$ be defined by $\omega (n,(a,b))= (a+2\pi k, b+2\pi k)$ and consider the quotient space $\mathbb T$ which can be identified  to $\mathbb C^\ast := \mathbb C\setminus 0$  by the map 
$$\mathbb T\ni \langle a, b\rangle \to z= e^{(b-a)+ ia}\in \mathbb C^\ast\setminus 0$$ with 
$\langle a, b\rangle$ denoting the class of $(a,b).$  In view of the equality $\delta^{\tilde f}_r(a,b)=\delta^{\tilde f}_r(a+2\pi k, b+2\pi k)$ one defines 

\begin{equation} \label {E4}
\delta^f_r(\langle a, b\rangle)= \delta^f_r(z):= \delta^{\tilde f}_r(a,b).
\end{equation}  
\vskip .1in

Consider  the commutative diagram with the multiplication by $t,$ (= $t\cdot$), inducing the linear isomorphism $\hat t _r(a,b)$
$$\xymatrix{
H_r(\tilde X)\ar[d]^{t\cdot= t_r} &\mathbb F_r(a,b)\ar[l]_{\supseteq}\ar[r]^{\pi_r(a,b)}\ar[d]^{t_r} &\hat\delta^{\tilde f}_r(a,b)\ar[d]^{\hat t_r(a,b)}\\
H_r(\tilde X) &\mathbb F_r(a+2\pi,b+2\pi)\ar[l]\ar[r] &\hat\delta^{\tilde f}_r(a+2\pi,b+2\pi)}$$
In view of  the equality  $t_r  ( \mathbb F_r(a,b) )=  \mathbb F_r(a+2\pi, b+2\pi)$  the subspaces 
\begin{equation}\label {E5}
\begin{aligned}
(\mathbb F_r) ' (\langle a, b\rangle):&= \sum_k \mathbb F'_r(a+2\pi k, b+2\pi k) \subseteq H_r(\tilde X)\\
\mathbb F_r (\langle a, b\rangle):&= \sum_k \mathbb F_r(a+2\pi k, b+2\pi k) \subseteq H_r(\tilde X)\\
\end{aligned}
\end{equation}
are 
 $\kappa[t^{-1}, t]$-submodules $H_r(\tilde X)$ . 
Denote by 
\begin{equation}\label {E6}
\hat\delta^f_r(\langle a, b\rangle ):= \bigoplus_k \hat\delta^{\tilde f}_r(a+2\pi k, b+2\pi k).
\end{equation} 
This vector space  equipped with the 
isomorphism $$\hat t_r= \hat t_r(\langle a, b\rangle): =\bigoplus_k    (\hat t_r(a+2\pi k,b+2\pi k): \hat\delta^f_r(\langle a, b\rangle )  \to \hat\delta^f_r(\langle a, b\rangle )$$ becomes  a free $\kappa[t^{-1}, t]$-module of rank $\delta^f_r(\langle a, b\rangle ).$

\vskip .1in

The  collection of  splittings  $\{s^{a,b}_r: \hat\delta(a,b)\to \mathbb F_r(a,b)\}$  which satisfy  
$t_r \cdot s^{a,b}_r= s^{a+2\pi, b+2\pi}_r\cdot \hat t_r(a,b)$  is called {\it compatible  splittings}. 
Clearly such collections exist. Indeed, it suffices to chose {\it splittings} only for  $\{(a,b)\in \supp (\delta^{\tilde f}_r), 0\leq a <2\pi\},$  observe that any $(a',b')\in \supp (\delta^{\tilde f}_r)$  is of the form $a'= a+2\pi k, b'= b+2\pi k$  for some integer $k\in \mathbb Z$ with $0\leq a<2\pi$   and take   $s_r(a',b'):= (\hat t_r)^k \cdot s_r(a,b) (\hat t_r)^{-k}.$  
Choose a collection of compatible splittings.

If one denotes by $$\supp (\delta^f_r)= \supp (\hat{\delta}^f_r):=\{\langle a, b\rangle \in \mathbb T \mid \delta^f_r(\langle a, b\rangle )\ne 0\},$$  
equivalently $$\{ z\in \mathbb C\setminus 0 \mid z= e^{ia + (b-a)}, \langle a, b\rangle\in \supp (\delta^f_r)\},$$
then the 
$\kappa$-linear map $I_r,$ constructed using {\it compatible splittings},  
$$I_r: \bigoplus _{\langle a, b\rangle\in \supp (\delta^f_r)}  \hat \delta^f_r(\langle a, b\rangle)\to H_r(\tilde X)$$  is  $\kappa[t^{-1},t]$-linear and  Proposition \ref{P37} implies 

\begin{corollary}\label{C38}\
\begin{enumerate}
\item  The composition of $\pi_r \cdot I_r$   with $\pi_r: H_r(\tilde X)\to H_r^N(X,\xi)$ the canonical projection is an isomorphism of free $\kappa[t^{-1},t]$--modules.
  \item The restriction of $I_r$ to $\oplus_{k\in \mathbb Z} \hat \delta^{\tilde f}_r(a+2\pi k, b+2\pi k)= \hat\delta^f_r(\langle a, b\rangle)$ 
 denoted by $I_r(\langle a, b\rangle)$ has the image contained in $\mathbb F_r(\langle a, b\rangle)$ and  $$\img I_r(\langle a, b\rangle)) \cap (\mathbb F_r)'(\langle a, b\rangle)= 0.$$ This implies that composed with the restriction of $\pi_r$ to $\mathbb F_r(\langle a, b\rangle)$ is injective  since $\pi_r\cdot I_r$ is an isomorphism. Moreover $\pi_r\cdot I_r (\langle a, b\rangle)$  is a split injective $\kappa[t^{-1},t]$-linear map whose image identifies to the free module $\mathbb F_r (\langle a, b\rangle)/ (\mathbb F_r '(\langle a, b\rangle).$ 
\item The collection of integers $\delta^f_r (\langle a, b\rangle )$ and the collection of free $\kappa[t^{-1}, t]$-modules $\hat \delta^f_r(\langle a, b\rangle)$ provide the configuration $\delta^f_r$ of points in $\mathbb T= \mathbb C\setminus 0$  and the configuration  $\hat \delta^f_r$ of  f.g. free $\mathbb C[t^{-1},t]-$modules, actually quotients of  split submodules of $H^N_r(X, \xi^f; \kappa[t^{-1},t]),$  as stated in introduction.  \end{enumerate}  
\end{corollary} 

\section {Proof of Theorem \ref{T1}.}

 Corollary \ref {C38}, Proposition \ref{P32} together with the formulae (\ref{E4}), (\ref{E5}) and (\ref{E6}) imply  items 1. and 2.
 
 Indeed Proposition \ref{P32} implies $H_r^N(X,\xi^f; \kappa[t^{-1},t])= H_r(\tilde X)/ (\mathbb I_{-\infty}(r) + \mathbb I^\infty(r))=H_r(\tilde X)/ T(H_r(\tilde X)).$
 The configuration $\delta^f_r$ is defined by (\ref{E4}) and the polynomial $P^f_r(z)$ by  $$P_r^f(z)= \prod_{z_i\in \supp (\delta^f_r)} (z-z_i)^{\delta^f_r(z_i)}.$$  
 
For $z=e^{ia +(b-a)}$ one takes $L(z)= \mathbb F_r (\langle a, b\rangle) $ and  $L'(z)= \mathbb F'_r (\langle a, b\rangle).$  
Corollary \ref {C38} items 2. and 3. imply that  the configuration $\hat \delta^f_r$ defined by equality (\ref{E6}) satisfies also $\hat \delta^f_r(z)= L(z)/ L'(z)$ as $\kappa[t^{-1}, t]$-modules and item 1. and 
the equality  (\ref{E4})  imply that the properties (a) and (b) in Theorem \ref{T1} item 2. are satisfied.

\vskip .1in 

Suppose $\kappa= \mathbb C.$ By choosing a $\mathbb C^[t^{-1},t]$--valued inner product on $H^N(X,\xi^f; \mathbb C[t^{-1}, t])$ and by using the von Neumann completion  described in subsection \ref{SS21}  one obtains from $\hat\delta^f_r$  the configuration $\hat{\hat \delta}^f_r(r)$ of closed Hilbert submodules of the $\mathcal N= L^\infty(\mathbb S^1)$-module $\overline{H^N_r(X,\xi: \mathbb C[t^{-1}, t])}$  as stated in Introduction.
This establishes Item 3.
 
The verifications of {\bf P1} and {\bf P2} will be discussed in \cite{Bu2}. They are first verified  for finite simplicial complexes and simplicial map then for Hilbert cube manifolds and  tame maps and then  concluded  for arbitrary continuous maps defined on arbitrary compact ANRs. This is why we have summarized in  subsection (\ref{SS23}) facts about Hilbert cube manifolds.


\end{document}